\documentclass{amsart}

\usepackage{amsthm}
\usepackage{amsmath}
\usepackage{amsfonts}
\usepackage{amssymb}
\usepackage{graphicx}

\usepackage{mathrsfs}
\usepackage[all]{xy}
\usepackage{hyperref}

\newtheorem{theorem}{Theorem}[section]

\newtheorem{corollary}[theorem]{Corollary}
\newtheorem{define}[theorem]{Definition}
\newtheorem{example}[theorem]{Example}
\newtheorem{lemma}[theorem]{Lemma}
\newtheorem{proposition}[theorem]{Proposition}

\newtheorem*{theoremnonum}{Theorem}

\newcommand{\FF}{\mathbb{F}}
\newcommand{\HH}{\mathbb{H}}
\newcommand{\PP}{\mathbb{P}}

\newcommand{\NN}{\mathbb{N}}
\newcommand{\CC}{\mathbb{C}}
\newcommand{\ZZ}{\mathbb{Z}}
\newcommand{\QQ}{\mathbb{Q}}
\newcommand{\Aff}{\mathbb{A}}

\newcommand{\thetachar}[4]{\theta\left[\begin{array}{c}#1\\#2\end{array}\right](#3,#4)}

\newcommand{\leftexp}[2]{{\vphantom{#2}}^{#1}{#2}}

\newcommand{\RA}{\mathcal{RA}}
\newcommand{\AM}{\mathcal{AM}}
\newcommand{\RM}{\mathcal{RM}}
\newcommand{\M}{\mathcal{M}}
\newcommand{\RRM}{\mathcal{R}^2\mathcal{M}}
\newcommand{\A}{\mathcal{A}}
\renewcommand{\P}{\mathcal{P}}
\newcommand{\RP}{\mathcal{RP}}
\newcommand{\J}{\mathcal{J}}
\newcommand{\RJ}{\mathcal{RJ}}
\newcommand{\RQ}{\mathcal{RQ}}
\newcommand{\C}{\mathcal{C}}
\newcommand{\RC}{\mathcal{RC}}
\renewcommand{\S}{\mathcal{S}}
\newcommand{\RS}{\mathcal{RS}}

\newcommand{\X}{\mathcal{X}}
\newcommand{\RX}{\mathcal{RX}}
\newcommand{\RGamma}{\mathcal{R}\Gamma}

\newcommand{\As}{\overline{\A}^s}
\newcommand{\At}{\overline{\A}^t}
\newcommand{\RAs}{\overline{\RA}^s}
\newcommand{\RAt}{\overline{\RA}^t}
\newcommand{\RMbar}{\overline{\RM}}
\newcommand{\Mbar}{\overline{\M}}
\newcommand{\RRMbar}{\overline{\RRM}}

\newcommand{\sh}{\mathscr}

\DeclareMathOperator{\Lines}{Lines}
\DeclareMathOperator{\WE}{WE}
\DeclareMathOperator{\Nm}{Nm}
\DeclareMathOperator{\Sp}{Sp}

\DeclareMathOperator{\diam}{diam}
\DeclareMathOperator{\Bl}{Bl}
\DeclareMathOperator{\Sym}{Sym}
\DeclareMathOperator{\Ram}{Ram}
\DeclareMathOperator{\mult}{mult}
\DeclareMathOperator{\sm}{small}

\DeclareMathOperator{\im}{im}
\DeclareMathOperator{\diag}{diag}
\DeclareMathOperator{\bg}{big}
\DeclareMathOperator{\smll}{small}

\begin{document}

\title{The Schottky problem in genus five}
\author{Charles Siegel}
\address{Kavli IPMU (WPI), The University of Tokyo, Kashiwa, Chiba 277-8583, Japan}
\email{charles.siegel@ipmu.jp}

\begin{abstract}
In this paper, we present a solution to the Schottky problem in the spirit of Schottky and Jung for genus five curves.  To do so, we exploit natural incidence structures on the fibers of several maps to reduce all questions to statements about the Prym map for genus six curves.  This allows us to find all components of the big Schottky locus and thus, to show that the small Schottky locus introduced by Donagi is irreducible.
\end{abstract}

\maketitle

\tableofcontents

\section*{Introduction}

The Schottky problem has a long history and is one of the first questions to be asked in complex curve theory: how can we characterize Jacobians among all abelian varieties? The question turns out to be quite subtle, and a completely satisfactory answer is still lacking, despite several characterizations existing and being a fertile ground for new techniques in studying the moduli of abelian varieties.

\subsection*{Historical Overview}

The oldest approach to the Schottky problem is via theta functions.  Certain theta functions give natural coordinates on $\A_g$, called thetanulls, and Schottky \cite{S} for genus four and Schottky and Jung \cite{SJ} for genus $g\geq 5$ conjectured equations for the locus of Jacobians.  We denote the locus these equations cut out by $\S_g$, the Schottky locus, and then the Schottky-Jung conjecture is that $\S_g=\J_g$.  Schottky and Jung's results implied only $\J_g\subset \S_g$, though that $\J_g$ is an irreducible component of $\S_g$ was shown by \cite{MR767196}.  A solution in the genus four case was announced in the late 1960s by Igusa \cite{MR0271117} and appeared independently in \cite{I} and \cite{MR701272}.  The method was to prove that $\S_4$ was an irreducible divisor on $\A_4$, and thus, must be identical with $\J_4$, an irreducible divisor contained inside it.  Unfortunately, this method cannot be employed directly in higher genus, as it relied on the fact that $\S_4$ is given by a single equation in theta nulls.

Another attempt to describe the Jacobians involves the dimension of the singular locus of the theta divisor.  The Riemann singularity theorem implies that this singular locus has dimension at least $g-4$, and this is a special condition on abelian varieties.  Andreotti and Mayer \cite{MR0220740} studied the locus $\AM_g$ of such abelian varieties and showed that it gives a weak solution to the Schottky problem, that is, the Jacobians are an irreducible component of $\AM_g$, but other components exist.  In genus four \cite{MR0572974} and five \cite{MR598683,MR1054527} the locus $\AM_g$ has been completely described as \[\AM_4=\theta_{null,4}\cup \J_4\mbox{ and } \AM_5=\A_1\times\A_4\cup\mathcal{BE}\cup \J_5,\] where $\mathcal{BE}$ is the locus of bielliptic Prym varieties and $\theta_{null,g}$ the locus of abelian varieties with a vanishing thetanull, with all the non-Jacobians contained in $\theta_{null,5}$.  This approach has not been carried out in higher genus fully, but singularities of theta divisors remains an active area of study \cite{MR2457734,MR1919422,MR2348405}.

A third approach has been through trisecants.  Fay's trisecant formula \cite{F} tells us that the natural embedding of $\J(C)/\pm 1\to \PP^{2^g-1}$ has a four dimensional family of trisecant lines.  Welters \cite{We2} proved that if $A/\pm 1\to\PP^{2^g-1}$ has a one parameter family of trisecant lines, then it must be a Jacobian.  The family of trisecants has an infinitesimal formulation as the KP hierarchy, and Novikov conjectured that an abelian variety is a Jacobian if and only if the corresponding theta function satisfies the KP hierarchy.  This was proved by Shiota \cite{Sh}.  The strongest result on trisecants is the recent proof, using difference equations, by Krichever \cite{krichever605625characterizing} that a single trisecant on $A/\pm 1$ suffices to determine that an abelian variety is a Jacobian.

\subsection*{This paper}

In \cite{MR903385,MR898055}, Donagi showed that the original Schottky-Jung conjecture was incorrect.  In particular, he showed that $\C$, the locus of intermediate Jacobians of cubic threefolds, is contained in $\S_5$.  Additionally, he offered a means of correcting this, and conjectured that nothing not accounted for by his methods appears in the Schottky locus.  However, even in genus five a proof was out of reach at the time, with some aspects of degenerations of dimension five abelian varieties remaining unclear until Izadi's thesis \cite{IzadiThesis}.

In this paper, we present a proof of

\begin{theoremnonum}
Inside of $\A_g$, we have \[\J_5=\S_g^{\smll}.\]
\end{theoremnonum}

The paper is organized as follows.  In section 1, we will gather results from the literature, recall definitions and set notation for the remainder of the paper.  Section 2 consists of the geometry and combinatorics of point and line configurations over $\FF_2$, and describes incidence structures on the fibers of several maps defined below.  Section 3 studies degenerations of abelian varieties of dimension five, using results from Izadi's thesis \cite{IzadiThesis} to complete the computation of the degree of the map $\beta_5$ begun in section 2 via incidence structures.  Section 4 consists of a study of the contracted loci for $\beta_5$, completing our picture of this map and allowing us to proceed to section 5, where we use this to compute the Schottky locus, and prove the main theorem.

\subsection*{Acknowledgments}
I would like to thank first of all Ron Donagi, who suggested this problem to me and advised me through the work as my Ph.D. thesis at the University of Pennsylvania.  Also, I am indebted to Sam Grushevsky, Gavril Farkas, Angela Ortega, Tony Pantev, Elham Izadi and Angela Gibney for conversations clarifying aspects of the problem as well as the geometry of the moduli spaces involved in this paper.  This work was supported by World Premier International Research Center Initiative (WPI Initiative), MEXT, Japan, and also by the University of Pennsylvania 

\section{Background}

\subsection{Moduli of curves and abelian varieties}

First, we will describe the various moduli spaces that we will need, all of which are covers of the spaces $\M_g$ and $\A_g$, the moduli space of genus $g$ compact Riemann surfaces and of $g$ dimensional principally polarized abelian varieties, respectively.

We will need first $\HH_g$, the Siegel upper half space of $g\times g$ symmetric complex matrices with positive definite imaginary part.  This space is the stack universal cover of $\A_g$, and has an action of $\Gamma_g=\Sp(2g,\ZZ)$ by 
\[\left(\begin{array}{cc}A&B\\C&D\end{array}\right)\cdot \Omega=(A\Omega+B)(C\Omega+D)^{-1}\]
 such that $\A_g\cong \HH_g/\Gamma_g$ as complex analytic spaces.

Other space we will need are defined by subgroups of $\Gamma_g$.  We set $\A_g(n)$ to be the quotient of $\HH_g$ by \[\Gamma_g(n)=\{\gamma\in \Gamma_g|\gamma\equiv 1 \mod n\}.\] This space parameterizes pairs $(A,\phi)$ where $A\in \A_g$ and $\phi\colon A[n]\to (\ZZ/n\ZZ)^{2g}$ is a symplectic isomorphism from the points of order $n$ on $A$ to $(\ZZ/n\ZZ)^{2g}$.  For theta functions, it will be essential to use $\A_g(n,2n)$, which is the quotient by \[\Gamma_g(n,2n)=\left\{\left.\left(\begin{array}{cc}A&B\\C&D\end{array}\right)\in\Gamma_g(n)\right|\diag(\leftexp{t}{A}C)\equiv\diag(\leftexp{t}{B}D)\equiv 0\mod 2n\right\}.\]

For the Prym map, we will need to fix the vector $v=(0,\ldots,0,1/2)\in\QQ^{2g}$, and then look at the groups
\begin{eqnarray*}
\RGamma_g(n)&=&\{\gamma\in\Gamma_g(n)|\gamma\cdot v\equiv v \mod (n\ZZ^{2g})\}\\
\RGamma_g(n,2n)&=&\{\gamma\in\Gamma_g(n,2n)|\gamma\cdot v\equiv v \mod (n\ZZ^{2g})\}
\end{eqnarray*}
corresponding to the moduli spaces $\RA_g(n)$ and $\RA_g(n,2n)$.  These spaces parameterize pairs $(A,\mu)$ where $A$ is in $\A_g(n)$ or $\A_g(n,2n)$, respectively, along with $\mu\in A[2]$ nonzero.

For $\M_g$, we will need all of the analogous spaces to the ones defined above for $\A_g$.  They can all be constructed via pullbacks: we use the Torelli map $\J_g\colon\M_g\to\A_g$ and take the fiber product with the projection to $\A_g$.  For example, $\RM_g$ is the space of pairs $(C,\mu)$ where $\mu$ is a nonzero point of order 2 on $\J(C)$.  We will need one additional space, $\RRM_g$, whose elements are ordered triples $(C,\mu,\nu)$ with $C\in\M_g$, $\mu,\nu$ distinct nonzero points of order $2$ on the Jacobian such that for any line bundle $L$ on $C$ with $L^{\otimes 2}\cong K_C$, we have 
\begin{equation}
\label{weilpairing}
h^0(C,L)+h^0(C,L\otimes \mu)+h^0(C,L\otimes \nu)+h^0(C,L\otimes \mu\otimes\nu)\equiv 0\mod 2.
\end{equation}
The left hand side is called the \textit{Weil pairing} and we will refer to pairs satisfying \ref{weilpairing} as orthogonal.  For abelian varieties that are not Jacobians, we can define the Weil pairing as the intersection pairing on $H_1(A,\ZZ/2\ZZ)$. (That these are the same is shown in \cite{MR0292836})

Finally, here we note another interpretation of $\RMbar_g$.  The data of a smooth curve $C$ and $\mu\in \J(C)[2]$ is equivalent to $\tilde{C}\to C$, an unramified double cover.  We can construct $\tilde{C}$ inside the total space of $\mu$ by fixing $\varphi:\mu^{\otimes 2}\to\sh{O}_C$ an isomorphism and then, on sufficiently small open sets $U\subset C$, we have \[\tilde{C}|_U=\{\sigma^2=1|\sigma\in \mu(U)\}.\]  We will use $(C,\mu)$ and $(C,\tilde{C})$ interchangeably, along with $(\tilde{C},\iota)$ where $\iota:\tilde{C}\to\tilde{C}$ is the involution corresponding to the double cover.

\subsection{Compactifications}

In this paper, we will need several partial compactifications of the moduli spaces we study, particularly $\M_g$, $\A_g$, $\RM_g$, $\RA_g$ and $\RRM_g$.

We will start with the boundary for $\A_g$ and $\RA_g$, which is simpler than that of $\M_g$ and its covers.  The space $\A_g$ has a distinguished compactification, the Satake compactification, which all other compactifications map to.  We will only need the corank 1 part, which we will denote by $\As_g$ and we note that $\partial \As_g\cong \A_{g-1}$.  There is a larger class of compactifications called toroidal compactifications \cite{MR2590897}.  We will only need the corank 1 part here as well, and for all toroidal compactifications, this is the same, so we write $\At_g=\Bl_{\A_{g-1}}\As_g$.  This has the property that $\partial\At_g$ is a divisor, and is isomorphic to $\X_{g-1}$, the universal Kummer variety over $\A_{g-1}$.  We will interpret it as parameterizing $\CC^\times$-extensions of abelian varieties of dimension $g-1$.

The space $\RA_g$ has a slightly more complex boundary, and for reference we reproduce results in \cite{MR903385} describing it.  There are both Satake and toroidal compactifications, for which we will again only need the corank 1 part.  In either case, there are three components, depending on the relationship between the point of order two, $\mu$, and the vanishing cycle $\delta$ taken modulo two:
\begin{itemize}
 \item $\partial^I\RAs_g\cong \A_{g-1}$ and $\partial^I\RAt_g\cong \X_{g-1}$ are the components where $\delta=\mu$.
 \item $\partial^{II}\RAs_g\cong\RA_{g-1}$ and $\partial^{II}\RAt_g\cong \RX_{g-1}$ (The fiber of $\RX_{g-1}\to \RA_{g-1}$ over $(X,\mu)\in \RA_{g-1}$ is the double cover of $X$ determined by $(\mu)^\perp$) are the components where $\delta\neq\mu$ and the two points are orthogonal.
 \item $\partial^{III}\RAs_g\cong \A_{g-1}$ and $\partial^{III}\RAt_g\cong \X_{g-1}$ are the components where $\delta$ and $\mu$ are not orthogonal.
\end{itemize}
In fact, we have

\begin{proposition}[{\cite[Proposition 2.3.2]{MR903385}}]
{\ \\}
The projection $\RAt_g\to \RAs_g$ is simply ramified along $\partial^{III}\RAt_g$ and unramified on the other boundary components.  (Note that the map $\partial^{III}\RAt_g\to \partial^{III}\RAs_g$ is given by $\RX_{g-1}\stackrel{\times 2}{\to}\RX_{g-1}\to \RA_{g-1}$ with the first map multiplication by two along the fibers)
\end{proposition}

For $\M_g$ and its covers, the boundary is somewhat more complex.  We will only be using partial compactifications contained in the Deligne-Mumford stable curve compactification \cite{MR0262240}.  In this compactification, we have $\delta_i$ for $1\leq i\leq \left\lfloor\frac{g-1}{2}\right\rfloor$ consisting of reducible curves with a component of genus $i$ meeting a component of genus $g-i$ at a point.  There is also a component $\delta_0$ consisting of irreducible curves of geometric genus $g-1$ with a single node\cite{MR1631825}.

The situation for $\RMbar_g$ is slightly more complex.  Each boundary component of $\Mbar_g$ splits into three components.  The reducible components $\delta_i$ give us $\partial_i$, $\partial_{g-i}$ and $\partial_{i,g-i}$, consisting of pairs $(C,\mu)$ where $\mu$ is supported on the curve of genus $i$, $g-i$ or both.  For $\delta_0$, we have $\partial^I$, $\partial^{II}$ and $\partial^{III}$ as with $\RAs_g$.  We also, for later reference, describe the corresponding double covers:

\begin{proposition}[{\cite[Examples 1.9]{MR1188194}}]
{\ \\}
The generic element of the boundary components of $\RMbar_g$ correspond to double covers as follows:
\begin{itemize}
 \item $\partial_i$ for $1\leq i\leq g-1$: the base curve is $C=C_i\cup_p C_{g-i}$ and the double cover is $\tilde{C}=C_{g-i}\cup_{p_1}\tilde{C}_i\cup_{p_2} C_{g-i}$.
 \item $\partial_{i,g-i}$ for $1\leq \left\lfloor\frac{g-1}{2}\right\rfloor$: the base curve is $C=C_i\cup_{p\sim q} C_{g-i}$ and $\tilde{C}=\tilde{C}_{g-i}\cup_{p_1\sim q_1,p_2\sim q_2} \tilde{C}_i$.
 \item $\partial^I$: the base curve is $C=X/p\sim q$ and the double cover is\\ $\tilde{C}=X_0\coprod X_1/p_0\sim q_1,p_1\sim q_0$.  This is called a \textit{Wirtinger double cover}.
 \item $\partial^{II}$: the base curve is $C=X/p\sim q$ and the double cover is\\ $\tilde{C}=\tilde{X}/p_0\sim q_0,p_1\sim q_1$ where $\tilde{X}\to X$ is an unramified double cover.  This is called an \textit{unallowable double cover}.
 \item $\partial^{III}$: the base curve is $C=X/p\sim q$ and the double cover is $\tilde{C}=\tilde{X}/\tilde{p}\sim \tilde{q}$ where $\tilde{X}\to X$ is a double cover ramified over $p$ and $q$.  This is called a \textit{Beauville double cover}.
\end{itemize}
\end{proposition}

\subsection{The Prym map}

In this section, we will define, describe and collect useful results about the Prym map.

Let $C$ be a genus $g$ curve and $\pi\colon\tilde{C}\to C$ be any morphism of curves.  Then we can define a map $\Nm_\pi\colon\J(\tilde{C})\to\J(C)$ by writing $D\in \J(\tilde{C})$ as $D=\sum n_P P$ and setting $\Nm_\pi(D)=\sum n_P \pi(P)$.  If $\pi$ is surjective, then so is $\Nm_\pi$.  This is especially useful in the case where $\pi$ is an unramified double cover.  In this case, the kernel is $g-1$ dimensional and has two components.  We define $\P_g(C,\tilde{C})=\ker^0\Nm_\pi$, the connected component of the identity of the kernel of the norm map, and call it the Prym variety of the double cover.  The principal polarization on $\J(\tilde{C})$ restricts to twice a principal polarization on $\P(C,\tilde{C})$\cite[Corollary 2]{MR0379510}, thus giving $\P(C,\tilde{C})$ the natural structure of a principally polarized abelian variety, giving us a map $\P_g:\RM_g\to\A_{g-1}$.  Additionally, we have:

\begin{lemma}[Mumford Sequence{\cite[Corollary 1]{MR0379510}}]
{\ \\}
Let $(C,\mu)\in \RM_g$.  Then we have a short exact sequence \[0\to \langle\mu\rangle\to \mu^\perp\to \P(C,\mu)[2]\to 0\] where we are taking the orthogonal complement with respect to the Weil pairing.

If $\nu\in \mu^\perp$ we will denote the image of $\nu$ in $\P(C,\mu)[2]$ by $\bar{\nu}$.
\end{lemma}

This allows us to define also $\RP_g\colon\RRM_g\to\RA_{g-1}$ by\\ $\RP_g(C,\mu,\nu)=(\P(C,\mu),\nu)$.

Unfortunately, as defined, the Prym map is not proper, so we must study the Prym map on the boundary of $\RMbar_g$.

\begin{proposition}[{\cite[Example 1.9]{MR1188194}}]
\label{BoundaryPrym}
{\ \\}
The Prym map extends to the general point of the boundary of $\RMbar_g$ as follows along each component:

\begin{itemize}
 \item $\partial^I$: For a Wirtinger double cover, $\P_g(C,\tilde{C})\cong \J_{g-1}(X)$.
 \item $\partial^{II}$: For an unallowable double cover, $\P_g(C,\tilde{C})$ is the $\CC^\times$-extension of $\P_{g-1}(X,\tilde{X})$ given by $p_0-p_1+q_0-q_1$ in $\P_{g-1}(X,\tilde{X})$
 \item $\partial^{III}$: The Prym of a Beauville double cover is a principally polarized abelian variety, and this and other ramified Prym maps are studied in \cite{MR2956039}.  It is isogenous to $\P_{g-1}(X,\tilde{X})$ but is not isomorphic to it.
 \item $\partial_i$ fits into a diagram:
 \[
\begin{xy}
(0,0)*+{\RM_i\times \M_{g-i}}="a";
(30,15)*+{\A_{i-1}\times \A_{g-i}}="b";
(0,15)*+{\partial_i}="c";
{\ar "c";"a"};
{\ar^{\P_g} "c";"b"};
{\ar_{\P_i\times \J_{g-i}} "a";"b"};
\end{xy}
 \]
 \item $\partial_{i,g-i}$ fits into a diagram:
 \[
\begin{xy}
(0,0)*+{\RM_i\times \RM_{g-i}}="a";
(40,15)*+{\partial\At_{g-1}}="b";
(40,0)*+{\A_{k-1}\times \A_{g-k-1}}="c";
(0,15)*+{\partial_{i,g-i}}="d";
{\ar_{\P_i\times\P_{g-i}} "a";"c"};
{\ar^{\P_g} "d";"b"};
{\ar "d";"a"};
{\ar "b";"c"};
\end{xy}
 \]
\end{itemize}
\end{proposition}

Unfortunately, $\P_g$ is only a rational map, it cannot be extended to all over $\RMbar_g$.  However, if we look at $\P_g^{-1}(\A_{g-1})=(\RMbar_g)^{allowable}$, then we have

\begin{theorem}[Prym is Proper {\cite[Theorem 1.1]{MR594627}}, {\cite[Proposition 6.3]{MR0572974}}]
\label{PrymisProper}
{\ \\}
For all $g$, $(\RMbar_g)^{allowable}\to \A_{g-1}$ is proper and for $g\leq 6$ it is surjective.
\end{theorem}

Finally, we will need a theorem of Shokurov, based on Mumford's analysis of the singularities of theta divisors:

\begin{proposition}[\cite{MR641127}, stated as it appears in \cite{MR1013156}]
\label{notbielliptic}
{\ \\}
Assume $g\neq 4$.  If $\P_g(C,\tilde{C})$ is a Jacobian, then $C$ is hyperelliptic, trigonal or is a plane quintic with an even point of order two.
\end{proposition}

In particular, this theorem implies that $C$ is not bielliptic, that is, it is not a branched cover of an elliptic curve, and this will be the main use we will have for it later.

\subsection{Fibers of the Prym map}

Any study of the fibers of the Prym map begins with the tetragonal construction of Donagi, announced in \cite{MR598683} and more fully developed in \cite{MR1188194}.  Let $\pi\colon\tilde{C}\to C$ be an unramified double cover and $f\colon C\to \PP^1$ be a degree four map.  Then we define \[f_*\tilde{C}=\{D\in\Sym^4\tilde{C}|\Nm_\pi(D)=f^{-1}(k)\mbox{ for some }k\in \PP^1.\] The involution $\iota$ on $\tilde{C}$ extends to $f_*\tilde{C}$, which also splits into two components, giving us two new towers $\tilde{C}_0\stackrel{\pi_0}{\to}C_0\stackrel{f_0}{\to}\PP^1$ and $\tilde{C}_1\stackrel{\pi_1}{\to}C_1\stackrel{f_1}{\to}\PP^1$ where $(C_i,\tilde{C}_i)\in\RMbar_g$ and $f_i$ is of degree four.  The key theorem is

\begin{theorem}[{\cite[Proposition 1.1]{MR598683}}]
{\ \\}
The tetragonal construction commutes with the Prym map: \[\P_g(C,\tilde{C})\cong \P_g(C_0,\tilde{C}_0)\cong \P_g(C_1,\tilde{C}_1).\]
\end{theorem}

This is a generalization of a theorem of Recillas identifying trigonal Pryms:

\begin{theorem}[The Trigonal Construction {\cite[Example 2.15(1)]{MR1188194},\cite{MR0480505}}]
\label{Trigconst}
{\ \\}
If $C\to \PP^1$ is tetragonal and $\tilde{C}\to C$ is the reducible double cover, then the tetragonal construction gives a copy of $\tilde{C}\to C\to \PP^1$ and also $\PP^1\cup \tilde{T}\cup\PP^1\to \PP^1\cup T\to \PP^1$ where $T$ is a trigonal curve and $\P_g(T,\tilde{T})\cong \J_{g-1}(C)$.  Additionally, this operation is a bijection, so every trigonal Prym is a tetragonal Jacobian.
\end{theorem}

and also we recall a theorem of Mumford's identifying hyperelliptic Pryms:

\begin{proposition}[\cite{MR0292836}]
\label{Hyperprym}
{\ \\}
The Pryms of hyperelliptic curves are products of Jacobians.
\end{proposition}

Now we will discuss the fibers of the Prym maps $\P_6$ and $\P_5$, as well as the behavior over certain special loci.  For a general genus six curve, there are exactly five $g^1_4$'s, and from this and the fact that any two points in the fiber of $\P_6$ are related by at most two applications of the tetragonal construction, we get

\begin{theorem}[{\cite[Theorem 2.1]{MR594627}\cite[Theorem 2.1]{MR598683}}]
\label{DS}
{\ \\}
The map $\RM_6\to \A_5$ is generically finite of degree 27 with Galois group $\WE_6$, the Weil group of the root system of type $E_6$, giving the general fiber the structure of the 27 lines on a cubic surface.
\end{theorem}

We can see this structure clearly over some special loci in $\A_5$.  However, first we need to perform some blowups.  Let $\RQ^0$ and $\RQ^1$ be the loci of plane quintic curves with an odd or even point of order two, respectively. (If $\mu$ is a point of order two on $\J(Q)$, its parity is the parity of the dimension of the space of global sections of $\mu\otimes g^2_5$.) Also let $\RM_6^{Trig}$ be the locus of genus six curves with a $g^1_3$, let $\C$ be the locus of intermediate Jacobians of cubic threefolds in $\A_5$.  Finally, set $\widetilde{\RM}_6$ to be the blowup of $\RM_6$ along $\RQ^0\cup\RQ^1\cup\RM_6^{Trig}$ and $\widetilde{\A}_5$ the blowup of $\A_5$ along $\J_5\cup \C$.  Then the Prym map $\tilde{\P}_6:\widetilde{\RM}_6\to\widetilde{\A}_5$.  This blown up Prym map is easier to analyze and gives more explicitly the structure of the tetragonal construction and the fibers.

\begin{proposition}[{\cite[Remark 4.5.1]{MR1188194}}]
\label{Tetspan6}
{\ \\}
Any two points in a fiber of $\tilde{\P}_6$ are related by at most two tetragonal constructions.
\end{proposition}

\begin{proposition}[{\cite[4.3.4]{MR1188194}}]
\label{Cubicfiber}
{\ \\}
The fiber of $\P_6$ over $X\in\C$ is the Fano surface of lines in $X$.  In $\widetilde{\A}_5$, the inverse image of $X\in\C$ is pairs $(X,H)$ where $H$ is a hyperplane in $\PP^4$.  Then, $\tilde{\P}_6^{-1}(X,H)=\{\ell|\ell\subset X\cap H\}$.
\end{proposition}

The above proposition exhibits the structure of the 27 lines very explicitly, and it follows from \cite[Theorem 2.1]{MR0472843} which uses the correspondence between conic bundle structures and lines on a cubic threefold to identify the intermediate Jacobians with the Pryms of discriminant curves.

Suppressing normal data, we can also describe the fiber over the Jacobian of a curve:

\begin{proposition}[{\cite[4.3.7]{MR1188194}}]
\label{Jacfiber}
{\ \\}
Let $C\in \M_5$.  Then $\tilde{\P}_6^{-1}(\J(C))$ contains one plane quintic with an odd point of order two, ten trigonal curves with double covers and sixteen Wirtinger double covers.
\end{proposition}

Another extremely useful result is 

\begin{proposition}[{\cite[Corollary 2.3]{MR598683},\cite[4.8]{MR1188194},\cite{MR690465}}]
\label{Prymbranchlocus}
{\ \\}
The ramification locus of $\P_6$ is mapped six to one onto the branch locus, which is the locus of intermediate Jacobians of quartic double solids.
\end{proposition}

Finally, for $\P_6$ we look over the boundary where we have

\begin{proposition}[{\cite[Corollary 6.5]{IzadiThesis}}]
\label{IzadiUnram}
{\ \\}
The restriction of the Prym map to $\partial^{II}\to \partial\At_5$ is generically unramified and finite.
\end{proposition}

Now, we recall the major results on $\P_5$:

\begin{proposition}[{\cite[Theorem 5.2]{MR1188194}}]
\label{genus5maps}
{\ \\}
\begin{enumerate}
 \item There exists an involution $\lambda:\RM_5\to \RM_5$ such that $\P_5\circ \lambda=\P_5$.
 \item There is a natural birational map $\chi:\At_4\dashrightarrow \overline{\RC}^0$, where $\overline{\RC}^0$ is the space of cubic threefolds with an even point of order two on the intermediate Jacobian, including nodal cubic threefolds as limits.
\end{enumerate}
\end{proposition}

Given these two maps, Donagi showed

\begin{proposition}[{\cite[Theorem 3.3]{MR598683}}]
{\ \\}
For $A\in\A_4$ generic, $\P_5^{-1}(A)/\lambda$ is isomorphic to the Fano surface of lines in $\chi(A)$, and $(C_1,\tilde{C}_1)$  and $(C_2,\tilde{C}_2)$ are tetragonally related if and only if the corresponding lines intersect.  Thus, any two points in $\P_5^{-1}(A)$ are related by at most two tetragonal constructions.
\end{proposition}

Before moving on, we will identify the fibers of $\P_5$ over Jacobians:

\begin{proposition}[{\cite[Theorem 5.14]{MR1188194}}]
\label{P5Jac}
{\ \\}
Let $B\in\M_4$ be a general curve of genus 4 and let $(X,\delta)=\chi(\J(B))$.  Then
\begin{enumerate}
 \item $X$ is a nodal cubic threefold
 \item The double cover of the Fano surface of lines in $X$ is reducible and each component is isomorphic to $\Sym^2B$
 \item $\P_5^{-1}(\J(B))$ is isomorphic to the double cover of the Fano surface of lines of $X$, with one component of trigonal curves $T_{p,q}$ and one of Wirtinger double covers $S_{p,q}$, for $(p,q)\in \Sym^2B$
 \item The tetragonal construction takes $S_{p,q}$ and $T_{p,q}$ to $S_{r,s}$ and $T_{r,s}$ if and only if $p+q+r+s$ is a special divisor on $B$, and $\lambda$ exchanges $S_{p,q}$ and $T_{p,q}$
 \item Two objects of $\P_5^{-1}(\J(B))$ are related by at most two tetragonal constructions.
\end{enumerate}
\end{proposition}

\subsection{Theta Functions}

Now, we shall describe one more major technical tool used in this paper.  The \textit{Riemann theta function} on $\HH_g\times\CC^g$ is given by \[\theta(\Omega,z)=\sum_{n\in\ZZ^g}\exp[\pi i(\leftexp{t}{n}\Omega n+2\leftexp{t}{n}z)].\] This function is periodic with respect to $\ZZ^g$ and is multiplied by an exponential factor with respect to $\Omega\ZZ^g$.  Thus, the zero locus of $\theta(\Omega,z)$ is periodic for $\ZZ^g\oplus \Omega\ZZ^g$ and gives a divisor $\Theta$ on $A$, the \textit{theta divisor}.

For $\epsilon,\delta\in\QQ^g$, we define a \textit{theta function with characteristics} to be \[\thetachar{\epsilon}{\delta}{\Omega}{z}=\exp[\pi i (\leftexp{t}{\epsilon}\Omega\epsilon+2\leftexp{t}{\epsilon}(z+\delta))]\theta(\Omega,z+\Omega\epsilon+\delta)\] which is essentially the translate of $\theta(\Omega,z)$ by $\Omega\epsilon+\delta$.  Evaluating at $z=0$, these are Siegel modular forms of weight $\frac{1}{2}$ and level $(4,8)$, so they are only well-defined on $\A_g(4,8)$, not on $\A_g$ itself.

We will also need $\theta_2[\epsilon](\Omega,z)=\thetachar{\epsilon}{0}{2\Omega}{0}$ where $\epsilon\in (\frac{1}{2}\ZZ/\ZZ)^g$, the second order theta functions, and $\theta\left[\begin{array}{cc}\epsilon&0\\0&1/2\end{array}\right](2\Omega,0)$ where $\epsilon\in (\frac{1}{2}\ZZ/\ZZ)^{g-1}$.  These are modular forms of weight $(2,4)$, and we we can use them to define maps $\alpha_g\colon\A_g(2,4)\to\PP(U_g)$ and $\beta_g\colon\RA_g(2,4)\to\PP(U_{g-1})$ where $U_g$ is the vector space of functions $(\ZZ/2\ZZ)^g\to \CC$ by setting $\alpha_g(\Omega)_\epsilon=\theta_2[\epsilon](\Omega,0)$ and
$\beta_g(\Omega)_\epsilon = \theta\left[\begin{array}{cc}\epsilon&0\\0&1/2\end{array}\right](2\Omega,0)$.

We get maps on $\A_g$ and $\RA_g$ by noting that $G_g=\Gamma_g/\Gamma_g(2,4)$ acts on $\A_g(2,4)$, $\RA_g(2,4)$ and $\PP(U_g)$ in compatible ways, so that if $\PP_g=\PP(U_g)/G_g$, we have $\alpha_g:\A_g\to\PP_g$ and $\beta_g:\RA_g\to \PP_{g-1}$.

The primary purpose of this paper will be to understand the fibers of $\beta_5$.  For this, we will use 

\begin{theorem}[Theta symmetry {\cite[Theorem 3.1]{MR898055}}]
\label{thetasym}
{\ \\}
Let $C\in\M_{g+1}$ be a curve of genus $g+1$ and let $\{0,\mu_0,\mu_1,\mu_2\}$ be a rank 2 isotropic subgroup of $\J_{g+1}(C)_2$ (thus, $\mu_2=\mu_0+\mu_1$).  For $i=0,1,2$ we have a Prym variety $P_i=P(C,\mu_i)\in\A_g$ and on it a uniquely determined semiperiod $\nu_i$, the image of $\mu_j$, $j\neq i$ in $P_i$.

The point $\beta(P_i,\nu_i)$ is independent of $i=0,1,2$.
\end{theorem}

We will also use the description of the fibers of $\beta_4$ that appears in the survey \cite{MR963063}

\begin{proposition}[{\cite[Theorem 5.3]{MR963063}}]
\label{beta4fiber}
{\ \\}
For $C\in\M_3$, the fiber $\beta_4^{-1}(\alpha_3(\J(C)))$ consists of two copies of the Kummer $K(\J(C))$, one contained in the interior of $\RM_4$ and the other the fiber over $\J(C)\in \A_3=\partial\As_4$ in the projection $\At_4\to\As_4$.
\end{proposition}

\subsection{The Schottky Loci}

We begin with the theorem that motivates this study:

\begin{theorem}[Schottky-Jung Identities \cite{S,SJ,MR0352108}]
\label{SJId}
{\ \\}
The following diagram commutes:

\[
\begin{xy}
(15,30)*+{\RM_g}="a";
(15,0)*+{\PP_g}="b";
(0,15)*+{\A_{g-1}}="c";
(30,15)*+{\RA_g}="d";
{\ar_{\P_g} "a";"c"};
{\ar^{\RJ_g} "a";"d"};
{\ar_{\alpha_{g-1}} "c";"b"};
{\ar^{\beta_g} "d";"b"};
\end{xy}
\]
\end{theorem}

Schottky and Jung noticed these relations between various theta functions and suggested that they can be used to describe Jacobians.  More precisely, we define 
\begin{eqnarray*}
\RS_g&=&\beta^{-1}_g(\im \alpha_{g-1})\\
\S_g^{\bg}&=&\{A|\exists \mu, (A,\mu)\in \RS_g\}\\
\S_g^{\smll}&=&\{A|\forall \mu, (A,\mu)\in \RS_g\}
\end{eqnarray*}
Schottky and Jung conjectured that $\S_g^{\bg}=\J_g$, the closure of the image of the Torelli map.  This is known not to be true, and the key result for that is Theta symmetry (Theorem \ref{thetasym})

Let $Q\subset \PP^2$ be a plane quintic curve and let $\mu,\nu$ be two points of order two on $\J(Q)$ such that $\mu$ is odd and $\nu$ is even.  Then $(\P(Q,\mu),\bar{\nu})\in\RS_5$ if and only if $(\P(Q,\nu),\bar{\mu})\in\RS_5$.  However, as mentioned above and proved in \cite{MR594627}, $\P(Q,\mu)$ is the intermediate Jacobian of a cubic threefold.  Additionally, $\P(C,\nu)$ is known to be a Jacobian of a curve (see the discussion in section 5 of \cite{MR898055}). Thus, $\C\subset\S_5^{\bg}$, and so Donagi introduced $\S_g^{\smll}$ to correct this, as only $\RC^0$ appears in $\RS_5$, not $\RC^1$, the locus of intermediate Jacobians of cubic threefolds with an odd point of order two.

There are two additional components that are not difficult to show lie in $\RS_5$.  There is the locus $\RA_1\times\A_4$, and in fact:

\begin{proposition}[{\cite[3.3.4]{MR903385}}]
{\ \\}
The Schottky locus $\RS_g$ contains $\A_{g'}\times \RA_{g''}$ for all $g'+g''=g$ (Note that if $g=0$, then $\RS_0=\emptyset$).
\end{proposition}

We will see below that no other product loci can be components of $\RS_5$.

There is also one boundary component, described by 

\begin{proposition}[{\cite[Theorem 3.3.1]{MR903385}}]
\label{Schottkyboundary}
{\ \\}
In the Satake compactification, we have \[\partial\overline{\RS}_g^s=\partial^I\RAs_g\cup \partial^{III}\RAs_g\cup i_{II}(\RS_{g-1})\] where $i_{II}$ is the inclusion of $\RA_{g-1}$ as $\partial^{II}\RAs_g$.
\end{proposition}

In \cite{MR963063} this result is strengthened to show that $\partial\overline{\RA}_g^t$ does not contain $\partial^{III}$ or $i_{II}(\RS_{g-1})$ as components, but only contains the points that are limits of other components.

\section{Line Configurations}

In this section, we will study a class of configurations that occur whenever there is a triality on a space, such as the tetragonal construction or theta symmetry.  Because we are working with trialities, everything in this section will be done over $\FF_2$, but the majority will work for $\FF_q$.

\begin{define}[Line Configuration]
{\ \\}
A line configuration $V$ over $\FF_2$ is a set $P_V$, called the points of $V$, along with a set $L_V$, called the lines of $V$, such that for all $\ell\in L_V$, we have $\ell\subset P_V$, $|\ell|=3$ and for all $\ell,\ell'\in L_V$, $|\ell\cap\ell'|\geq 2$ implies that $\ell=\ell'$.
\end{define}

We can construct a large class of examples, which we will call \textit{algebraic line configurations}.  These are defined by starting with $V\subset\PP^n$ over $\FF_2$ a projective variety.  Then we set $P_V$ to be the $\FF_2$-points of $V$ and $L_V$ to be the projective lines over $\FF_2$ contained in $V$.  The first interesting example is $V=\PP^2$, the Fano plane:

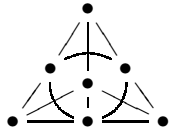
\[
\begin{xy}
(10,0)*+{\bullet}="a1";
(20,0)*+{\bullet}="a2";
(30,0)*+{\bullet}="a3";
(20,5)*+{\bullet}="b1";
(15,7)*+{\bullet}="c1";
(25,7)*+{\bullet}="c2";
(20,15)*+{\bullet}="d1";
{\ar@{-} "a1";"a2"};
{\ar@{-} "a2";"a3"};
{\ar@{-} "a1";"c1"};
{\ar@{-} "c1";"d1"};
{\ar@{-} "a3";"c2"};
{\ar@{-} "c2";"d1"};
{\ar@{-} "a1";"b1"};
{\ar@{-} "b1";"c2"};
{\ar@{-} "a3";"b1"};
{\ar@{-} "b1";"c1"};
{\ar@{-} "a2";"b1"};
{\ar@{-} "b1";"d1"};
{\ar@/^.5pc/@{-} "a2";"c1"};
{\ar@/^.5pc/@{-} "c1";"c2"};
{\ar@/^.5pc/@{-} "c2";"a2"};
\end{xy}
\]

we define a morphism of configurations $f\colon V\to W$ to be an injection $f\colon P_V\to P_W$ such that for each line $\{p_0,p_1,p_\infty\}\in L_V$, the image $\{f(p_0), f(p_1), f(p_\infty)\}$ is a line in $L_W$.  A \textit{subconfiguration} is then the image of a morphism.  We say that two points are collinear if they lie on a subconfiguration isomorphic to $\PP^1$ and that two lines are coplanar if they lie on a subconfiguration isomorphic to $\PP^2$.

\begin{define}[$V$-Configuration]
{\ \\}
Let $V$ be a line configuration.  Then another line configuration $W$ is a \\ {$V$-configuration} if for each $p\in P_W$, there exists a bijection $\phi_p\colon \{\ell\in L_W|p\in \ell\}\to P_V$ such that $\ell,\ell'$ are coplanar if and only if $\phi_p(\ell)$ and $\phi_p(\ell')$ are collinear.
\end{define}

\begin{example}
{\ \\}
\begin{enumerate}
 \item $\PP^n$ is a $\PP^{n-1}$-configuration.
 \item If $V$ is a collection of $n$ points with $L_V=\emptyset$, then $(\PP^1)^n$ is a $V$-configuration.
 \item For some $V$, we can construct examples of multiple fundamentally distinct $V$-configurations.  For instance, if $V$ consists of five points and no lines, the previous example says that $(\PP^1)^5$ is a $V$-configuration.  However, if $S$ is a smooth cubic surface, then we can define a configuration $W$ with $P_W=\{\ell\subset S|\ell\mbox{ is a line}\}$ and with $L_W$ the set of triples of coplanar lines.  It is classical that given a line on a cubic surface, it is contained in exactly five coplanar triples, no set of which are configured as a Fano plane.
\end{enumerate}
\end{example}

Now, to help us to describe the properties of line configurations, we use the \textit{incidence graph}.  This is the graph $\Gamma_V$ whose vertices are the points of $P_V$ and two vertices $p,q\in P_V$ are connected by an edge if and only if they are collinear.  We fix a metric on $\Gamma_V$ such that each edge has length 1, and give properties to $V$ from the properties of the metrized graph $\Gamma_V$, for instance, we can speak of connected configurations or the diameter of a configuration.

Given $\Gamma_V$, we can define numerical invariants of a configuration.  For each $p\in P_V$, we define $V_i(p)=\{q\in V|d(p,q)=i\}$, the points that are first reached after passing along $i$ lines from $p$. For points $p\in P_V$, and $q\in V_i(p)$, we define $V_{i,j}(p,q)=V_j(p)\cap V_1(q)$, the points distance $j$ from $p$ which are adjacent to $q$.  Note that $V_{i,j}(p,q)$ is empty unless $j$ is $i-1$, $i$ or $i+1$.  We denote the cardinalities of these sets by $v_i(p)$ and $v_{i,j}(p,q)$ (by convention, we set these numbers to be zero if $i$ or $j$ is negative), and call a line configuration \textit{symmetric} if the $v_i(p)$ and $v_{i,j}(p,q)$ don't depend on $p$ and $q$, in which case we will denote $v_i=v_i(p)$ and $v_{i,j}=v_{i,j}(p,q)$.

\begin{proposition}
\label{propnumerics}
{\ \\}
Let $V$ be a connected symmetric line configuration.  Then
\begin{enumerate}
 \item For all $i$, $v_1=v_{i,i-1}+v_{i,i}+v_{i,i+1}$
 \item For all $i$, $v_1v_i=v_{i-1,i}v_{i-1}+v_{i,i}v_i+v_{i+1,i}v_{i+1}$
 \item We have $v_0=1$, $v_{0,0}=0$, $v_{0,1}=v_1$ and $v_{1,0}=1$.
\end{enumerate}

Let $W$ be a connected symmetric $V$-configuration.  Then

\begin{enumerate}
 \item[(4)] $w_{2,2}\geq w_{2,1}$
 \item[(5)] $w_1=2|P_V|$
 \item[(6)] $w_{1,1}=2v_1+1$
 \item[(7)] If, additionally, $v_3=w_3=0$, then either $w_2=|P_V|$ or $w_2=4v_2$.
\end{enumerate}
\end{proposition}

For convenience, throughout the proof we will set $\Lines_W(p)=\{\ell\in L_W|p\in\ell\}$, the set of lines through $p$ in $W$.

\begin{proof}
Let $V$ be a connected symmetric line configuration and $W$ a connected symmetric $V$-configuration.

\begin{enumerate}
 \item Fix $p\in V$, $i\in\NN$, $q\in V_i(p)$.  Then 
 	\begin{eqnarray*}
 	&&V_{i,i-1}(p,q)\cup V_{i,i}(p,q)\cup V_{i,i+1}(p,q)\\
 	&=&(V_{i-1}(p)\cap V_1(q))\cup(V_{i}(p)\cap V_1(q))\cup(V_{i+1}(p)\cap V_1(q))\\
 	&=&(V_{i-1}(p)\cup V_i(p)\cup V_{i+1}(p))\cap V_1(q)\\
 	&=&V_1(q).
 	\end{eqnarray*}
 \item Fix $p\in V$, $i\in\NN$.  Let $X=\{(a,b)|a\in V_i(p),b\in V_1(a)\}$.  Then $|X|=v_iv_1$.  But also, 
 \begin{eqnarray*}
 X&=&\{(a,b)|a\in V_i(p)\cap V_1(b),b\in \cup_j V_j(p)\}\\
 &=&\cup_j \{(a,b)|a\in V_{j,i}(p,b),b\in V_j(p)\}
 \end{eqnarray*} and so $|X|=\sum_j v_{j,i}v_j$.
 \item These all follow directly from the definitions.
 \item Fix $q\in W_2(p)$.  We prove that no line containing $q$ contains two points of $W_1(p)$.  As each line consists of $3$ points, this implies that $w_{2,2}\geq w_{2,1}$.  Let $a,b\in W_1(p)$ and assume that there is a line $\ell\in L_W$ such that $a,b,q\in\ell$.  As $a,b\in W_1(p)$, there exist lines $m_1,m_2$ through $p$ containing $a,b$ respectively.  But then, $m_1,m_2$ must be coplanar, and so there is a line $m$ containing $p,q$, so $q\in V_1(p)\cap V_2(p)=\emptyset$, a contradiction.
 \item For each line $\ell\in \Lines_W(p)$, fix a bijection $\ell\to \PP^1(\FF_2)$ such that $p$ is mapped to $\infty$.  Then $\cup_{\ell\in\Lines_W(p)} \ell\setminus \{p\}=\coprod_{\ell\in \Lines_W(p)}\Aff^1(\FF_2)$, and this has cardinality twice the number of lines, $2|P_V|$.
 \item Fix $p\in P_W$, $\ell$ a line through $p$, $p'\in \ell$ distinct form $p$.  Through $p'$, there are $|P_V|$ lines.  One is $\ell$, $v_1$ of them are coplanar with $\ell$, and the rest are not.  Each coplanar line consists of $n$ points in $V_1(p)$ that are not $p'$, but there are also $n-1$ points of $\ell$ in $V_1(p)$ other than $p'$, and so $w_{1,1}=2v_1+1$.
 \item Fix $p\in W$.  Let $\tilde{X}$ be the set of triples $(\ell,m,q)$ in $\Lines_W(p)\times L_W\times W_2$ such that $p\in \ell$, $\ell\cap m\neq \emptyset$ and $q\in m$.  There is a natural map $\tilde{X}\to W_2$ which is surjective.  Then
\begin{eqnarray*}
 |\tilde{X}|&=&w_2\cdot |\mbox{fiber}|\\
 &=&w_2\cdot |\{\mbox{paths to }q\in W_2\mbox{ from }q\}|\\
 &=&w_2\cdot |\{\mbox{points of} W_1(p) \mbox{connected to }p\}|\\
 &=&w_2\cdot |W_1(p)\cap W_1(q)|\mbox{ for }q\in W_2(p)\\
 &=&w_2w_{2,1}
\end{eqnarray*}

Now, we also have a map $\tilde{X}\to \Lines_W(p)\times W_2(p)=A\coprod B$, where the fiber over $A$ has cardinality 1 and over $B$ has cardinality 0.  These are the only possibilities, because if there were two, then we get a plane and $q\in W_1(p)$.  So $|A|=|\tilde{X}|=w_2w_{2,1}$, and $|A|+|B|=|\Lines_W(p)\times W_2(q)|=|P_V|w_2$, so $|B|=w_2(|P_V|-w_{2,1})$.

But, as $W$ is symmetric, we can see that $|V|=|P_V|\alpha$, where $\alpha$ is the number of lines in $\Lines_W(p)$ that don't have a line connecting them to $q$.  Then $\alpha=|P_V|-|\{\mbox{lines in }\Lines_W(p)\mbox{ connected to }q\}|$, which is $\alpha=|V|-w_{2,1}$, so $|B|=|P_V|(|P_V|-w_{2,1})$.

So, $w_2(|P_V|-w_{2,1})=|P_V|(|P_V|-w_{2,1})$, implying that either $w_2=|P_V|$ or else $w_{2,1}=|P_V|$.  In the latter case, the previous parts of this proposition along with the hypothesis that $\diam(W)\leq 2$, implies that $w_2=4v_2$.
\end{enumerate}
\end{proof}

We can now apply these numerics to another family.  Let the zero locus of $x_1^2+x_2^2+x_1x_2+x_3x_4+\ldots+x_{2n-1}x_{2n}$ in $\PP^{2n-1}(\FF_2)$ be denoted by $Q_{2n}^-$.  This is the smooth quadric of Witt defect 1 in $\PP^{2n-1}$ over $\FF_2$.  We have actually already seen it in an above example: if $n=2$, then the zero locus is five points, no three collinear, and if $n=3$, we get 27 points, each of which lies on 5 lines (and thus is adjacent to 10 points), which gives the configuration of lines on a smooth cubic surface.  The latter fact also appears as an isomorphism $WE_6\cong O^-_6(\FF_2)$, the orthogonal group preserving the form $x_1^2+x_2^2+x_1x_2+x_3x_4+x_5x_6$.

For all $n$, in fact, we can see that $Q_{2n+2}^-$ is a $Q_{2n}^-$-configuration, by noting that they are homogeneous varieties, and so we only need to look at the point $[0\colon\ldots\colon0\colon1]$, where everything is easy to compute.  Additionally, for $n\geq 3$ these are all symmetric and have diameter two, so all of the numerical conditions above apply.

\begin{lemma}
\label{119forbetaq6}
{\ \\}
Let $n\geq 2$.  Then the only connected, symmetric $Q_{2n}^-$-configuration of diameter 2 is $Q_{2n+2}^-$.
\end{lemma}

\begin{proof}
We set $V=Q_{2n}^-$ for some fixed $n\geq2$ and $W$ a connected, symmetric $V$-configuration.

As we are assuming diameter 2, we have to determine $w_i$ and $w_{i,j}$ for $0\leq i,j\leq 2$.  Proposition \ref{propnumerics}(3) determines $w_0$, $w_{0,0}$, $w_{1,0}$ and that $w_{0,1}=w_1$.  Then Proposition \ref{propnumerics}(5) says that $w_1=2|Q_{2n}^-|=2(2^{n-1}(2^n-1)-1)=2^n(2^n-1)-2=2^{2n}-2^n-2$.

This leaves $w_{1,1},w_{2,1},w_{1,2}, w_{2,2}$ and $w_2$.  Proposition \ref{propnumerics}(6) gives $w_{1,1}=2(2^{2n-2}-2^{n-1}-2)=2^{2n-1}-2^n-4$.  Then Proposition \ref{propnumerics}(1) says that $2^{2n}-2^n-2=1+w_{1,1}+w_{1,2}$ and $2^{2n}-2^n-2=w_{2,1}+w_{2,2}$, which we solve for $w_{1,2}=2^{2n}-2^n-2-1-(2^{2n-1}-2^n-4)=2^{2n}-2^{2n-1}+1$.

Substitution of \ref{propnumerics}(1) into \ref{propnumerics}(2) and cancellation gives us the relation $w_iw_{j,i}=w_jw_{i,j}$.  And thus, $w_1w_{2,1}=w_2w_{1,2}$.  So $(2^{2n-1}-2^n-4)w_{2,1}=w_2(2^{2n}-2^n-2-1-(2^{2n-1}-2^n-4))$.  Using \ref{propnumerics}(7), we know that $w_2=(2^{n-1}(2^n-1)-1)$ or $4((2^{n-1}(2^n-1)-1)-1-2((2^{n-2}(2^{n-1}-1)-1)))=2^{2n}$.  The former case leads to $w_{2,2}<w_{2,1}$ contradicting \ref{propnumerics}(4), and so we must have the latter case.  And finally, this determines $w_{2,1}$ using $w_1w_{2,1}=w_2w_{1,2}$ and $w_{2,2}$ from $w_1=w_{2,1}+w_{2,2}$.

In particular, $|W|=|Q_{2n+2}^-|=2^n(2^{n+1}-1)-1$.

We now prove that these $2^n(2^{n+1}-1)-1$ points must be connected by lines in a unique way to satisfy the conditions of the lemma.  We begin with a point of the configuration $w\in W$.  As $w_{2,1}=w_{2,2}=2^{n-1}(2^n-1)-1=|Q_{2n}^-|$, each line not containing $w$ must contain one point in $W_1(w)$ and two points of $W_2(w)$.  Thus, a point in $W_2(w)$ is given by a point (not $w$) on each line containing $w$.

The same holds for any point $w'\in W_1(w)$, and there must be a function $W_1(w)\to W_1(w')$ that preserves incidence and takes the set of choices at $w$ to those at $w'$.  This will dictate the data of which half (we note that $w_{1,2}=\frac{1}{2}w_2$, and so each point of $W_1(w)$ is colinear with half of the points of $W_2(w)$) of the points of $W_2(w)$ are connected to $w'$.  As this works for each $w'\in W_1(w)$, there is only one such choice that will work globally, up to automorphisms of $Q_{2n}^-$ and involutions switching non-$w$ points on a line.
\end{proof}

So to identify a $Q_{2n}^-$-configuration, we must merely check symmetry and diameter.

\begin{proposition}
\label{Prop:BetaIsConfig}
{\ \\}
The general fiber of $\beta_5\colon\RA_5\to \PP^{15}/G_4$ is a $Q_6^-$-configuration.
\end{proposition}

\begin{proof}
Let $(A,\mu)\in \RA_5$ be a general point.  Then $A$ is the Prym variety of 27 distinct curves $(X_i,\nu_i)\in \RM_6$, and $\mu$ lifts to two distinct points of order two on $\J(X_i)$, $\mu_i^0,\mu_i^1$.  Then theta symmetry implies that, again generically, $(\P(X_i,\mu_i^j),\bar{\nu}_i)$ are distinct points in the fiber of $\beta_5$.  Thus, $(A,\mu)$ lies on 27 triples $(A,\mu)$, $(\P(X_i,\mu_i^0),\bar{\nu}_i)$, $(\P(X_i,\mu_i^1),\bar{\nu}_i)$, which are lines over $\FF_2$, form the configuration $Q_6^-$, by Theorem \ref{DS}, and so the fiber is a $Q_6^-$-configuration.
\end{proof}

\begin{lemma}
{\ \\}
The connected components of the $Q_6^-$-configuration on the fibers of $\beta_5$ has diameter 2.
\end{lemma}

\begin{proof}
If we blow up $\RP_6:\RRM_6\to \RA_5$ to a finite map $\widetilde{\RP}_6$, then Proposition \ref{Cubicfiber} says that the fiber over $(X,\mu)\in\RC^0$, with the normal data suppressed for ease of notation, consists of 54 plane quintics with an odd and an even point of order two marked.  Theta symmetry then says that the Pryms with respect to the even points with the images of the odd points marked are in the same fiber of $\beta_5$, and all of these are the Jacobians of curves.  By Proposition \ref{Jacfiber}, the fiber of a Jacobian consists of quintic curves, whose Pryms are cubic threefolds and Jacobians, trigonal curves, whose Pryms are Jacobians, and singular curves, whose Pryms are Jacobians and degenerate abelian varieties in $\partial^I\RAt_5$.

If we start with a degenerate abelian variety, then we must only have singular curves lying over it.  Proposition \ref{BoundaryPrym} tells us that these are either in $\partial^{II}$ or $\partial_{i,6-i}$.  But the latter component only gives products, and so generically all 27 preimages under $\P_6$ are in $\partial^{II}$.  Thus, any points of $\RAt_5$ related to a point of $\partial^I \RAt_5$ is either in $\partial^I\RAt_5$ or is a Jacobian, and so no new types of points occur after the second iteration of theta symmetry.

Finally, we show that the objects obtained previously don't constitute anything new.  First, the Jacobians obtained from the chosen Jacobian in the second step.  These Jacobians all must have already been obtained from the cubic threefold, because each Jacobian is the Prym of a unique quintic curve, and will thus determine a cubic threefold.  However, generically no two cubic threefolds can be related in this manner, because the incidence is contained in the fibers of $\beta_5$, which induces a birational isomorphism between $\RC^0$ and $\alpha(\A_4)$, and so, generically, must be injective.  And last, we have the objects obtained from a given degenerate abelian variety.  We have twenty-seven nodal curves, and each occurs as a Wirtinger and as a $\partial^{II}$ double cover, which we then take the Pryms of.  The Wirtingers must be Jacobians we've already seen, by the above, and so the $\partial^{II}$ double covers would be obtained by applying theta symmetry on those Jacobians first.  Thus, the $Q_6^-$-configuration must have diameter 2.
\end{proof}

And so, by Lemma \ref{119forbetaq6}, the map $\beta_5$ must have degree a multiple of 119.

\section{Degenerations of Abelian varieties}

In this chapter, we will show that the $Q_6^-$-configuration on the fibers of $\beta$ is connected.  This amounts to computing the degree of $\beta$, which we will do by studying certain degenerations.  We will spend this chapter studying the structure of $\partial^{II}\RAt_5$, the degenerations with vanishing cycle orthogonal but not identical to the marked semiperiod.

\begin{lemma}[{\cite[Corollary 3.2.3 and Lemma 3.3.6]{MR903385}}]
\label{Lemma:dIIbeta}
{\ \\} The extension of $\beta_g\colon\RAt_g\to \PP(U_{g-1})/G_{g-1}$ to $\partial^{II}\RAt_g$ is $\beta_{g-1}$ and the diagram:
\[\begin{xy}
(0,0)*+{\PP(U_{g-2})/G_{g-2}}="a";
(30,0)*+{\PP(U_{g-1})/G_{g-1}}="b";
(0,15)*+{\RA_{g-1}}="c";
(30,15)*+{\RAt_g}="d";
{\ar "a";"b"};
{\ar^{i_{II}} "c";"d"};
{\ar_{\beta_{g-1}} "c";"a"};
{\ar^{\beta_g} "d";"b"};
\end{xy}\]
is Cartesian.
\end{lemma}

This result tells us that if we start with a point in $\partial^{II}\RAt_g$ then the whole fiber over its image is in $\partial^{II}\RAt_g$, and so we focus our attention there:

\begin{proposition}
\label{Prop:ThetaSpan}
{\ \\}
Any two points in a fiber of $\beta_5$ inside $\partial^{II}\RAt_5$ are related by a sequence of theta symmetries.
\end{proposition}

\begin{proof}
By Lemma \ref{Lemma:dIIbeta}, we are working with the map $\beta_4\colon\RA_4\to \PP^7/G_3$.  Now, fix a curve $C\in\M_3$.  Then the fiber of $\beta_4$ over $\alpha_3(\J(C))$ is $\Bl_0K(\J(C))\cup K(\J(C))$ by Proposition \ref{beta4fiber}, with the first component consisting of genus 4 Jacobians and the second of degenerate abelian varieties in $\partial^I\RAt_4$.  Over $(B,\mu)\in \RJ_4$, the fiber of $\RP_5$ is a double cover of $\Sym^2B\cup \Sym^2B$, with the first component consisting of trigonal curves and the second consisting of Wirtinger curves, as in Proposition \ref{P5Jac}.  Thus, for each $(B,\tilde{B})$, we get a map from the double cover of $\Sym^2B$, $\widetilde{\Sym^2B}\to \Bl_0K(\J(C))$, because theta symmetry takes trigonal curves to trigonal curves and Wirtinger curves to Wirtinger curves.  Each of these maps has two dimensional image, and they're all nonisomorphic, and thus distinct, so the dimension of the union is at least three, so surjectivity follows.
\end{proof}

Now that we have a locus where we know that theta symmetry spans the fibers, we need to identify theta symmetry over that locus:

Proposition \ref{IzadiUnram} tells us that the fiber in $\partial^{II}\RAt_5$ has the $Q_6^-$-configuration structure we expect, and more so, the fact that theta symmetry spans the fibers implies connectivity, so the fiber is $Q_8^-$.  It remains, though, to check that $\beta_5$ is generically unramified on $\partial^{II}\RAt_5$.

We have a tower of maps $\RRMbar_6\stackrel{\RP_6}{\to}\RAt_5\stackrel{\beta_5}{\to}\PP^{15}/G_4$ and as each map is generically finite, if we blow up each locus with positive dimensional fibers, and the blow up again so that the maps are all well defined, the maps will be finite.  We denote the blowups and blown up maps by $\widetilde{\RRM}_6\to \widetilde{\RA}_5\to\widetilde{\PP}$.  Set $X=\widetilde{\RA}_5\times_{\PP}\widetilde{\RA}_5$.  For each point $(x,y)\in X$, we can associate the distance in $\Gamma_{Q_8^-}$, or $\infty$ if they are not connected.  This gives a decomposition $X=I_1\cup I_{54}\cup I_{64}\cup I$, where $I_1$ is the diagonal, $I_{54}$ are theta related pairs, $I_{64}$ the pairs theta related to a common point and $I$ the pairs in the same fiber but not theta related.  Our goal is to show that $I=\emptyset$.

We will focus on studying $I_{54}$.  It fits into a diagram
\[
\begin{xy}
(0,0)*+{I_{27}}="a";
(0,30)*+{I_{54}}="b";
(30,30)*+{\widetilde{\RA}_5}="c";
{\ar_{2:1} "b";"a"};
{\ar^{\tilde{\beta}} "b";"c"};
{\ar_{\tilde{\P}} "a";"c"};
\end{xy}
\]
where $I_{27}$ is the pullback of the Prym map to $\widetilde{\RA}_5$ from $\A_5$, and $I_{54}\to I_{27}$ is the natural double cover.  Differentiating these maps, we can see that $\Ram\tilde{\beta}=\Ram\tilde{\P}$.

\begin{lemma}
\label{lemma:ram1}
{\ \\}
Let $f\colon X\to Y$ be a finite morphism of smooth varieties and $\bar{f}\colon X\times_Y X\to X$ be the fiber product of $f$ with itself.  Then if $p$ is a ramification point of $f$ and $q$ is a nonramification point in the same fiber, $(p,q)$ is a ramification point of $\bar{f}$.
\end{lemma}

\begin{proof}
Set $b=f(q)=f(p)$.  As $p$ is a ramification point and $q$ is not, we have $df_q$ an isomorphism and $df_p$ not.  Locally, this means that we have an isomorphism of $X$ and $Y$ near $q$ but not near $p$, and so the map $d\bar{f}_{(p,q)}\colon T_{p,q}(X\times_YX)\to T_pX$ is not an isomorphism.
\end{proof}

\begin{lemma}
\label{lemma:ram2}
{\ \\}
Let $p\in \partial^{II}\widetilde{\RA}_5$.  Then if $\beta_5$ is ramified at $p$, $I_{54}$ is ramified over $p$.
\end{lemma}

\begin{proof}
Let $p\in \partial^{II}\widetilde{\RA}_5$ be a ramification point of $\tilde{\beta}$.  Then either there exists a $q$ such that $(p,q)\in I_{54}$ is unramified, or else not.  If there is such a $q$, then this follows from Lemma \ref{lemma:ram1}.

If not, then we must have a fiber of the Prym map which is totally ramified.  This would correspond to a cubic surface such that every line has multiplicity at least two, but which is reduced and irreducible.  This cannot happen, as if there is a line of multiplicity two, then the residual intersection with the tangent plane along that line must be a line of multiplicity one, so every line would have multiplicity three, which does not occur.
\end{proof}

More useful than the lemma is the contrapositive: that if $I_{54}$ is unramified over $p$, then $\beta_5$ is unramified at $p$.

\begin{proposition}
{\ \\}
The map $\beta_5$ is generically unramified on $\partial^{II}\RAt_5$.
\end{proposition}

\begin{proof}
Lemma \ref{lemma:ram2} tells us that a point is a ramification point for $\beta_5$ only if it is a branch point of $\tilde{\beta}$, which is the same as being a branch point over $\tilde{\P}$.  But Izadi's Theorem \ref{IzadiUnram} tells us that the Prym map is generically unramified over $\partial^{II}\RAt_5$, and so $I_{54}\to \widetilde{\RA}_5$ is generically unramified over $\partial^{II}\widetilde{\RA}_5$, and so, $\beta_5$ is as well.
\end{proof}

Thus, we have the main result of this section:

\begin{theorem}
\label{thm:betatheta}
{\ \\}
the fibers of $\beta_5\colon \RAt_5\to \PP^{15}/G_4$ are exactly the orbits of theta symmetry.
\end{theorem}

\section{Contracted loci}

Now, we've identified the fibers of $\beta_5$ with the orbits of theta symmetry, which gives us the degree of this generically finite map.  What remains is to understand the locus where $\beta_5$ has infinite fibers.  As each abelian variety has finitely many points of order two, and as $\alpha_g$ is finite and generically injective for all $g$ \cite[Proposition 1]{MR1305875}, any infinite fiber must arise from an infinite fiber of the Prym map.

By Proposition \ref{Tetspan6} we know that given $(C,\mu)\in \RM_6$, we obtain the rest of the fiber containing $(C,\mu)$ by iterating the tetragonal construction, so long as we allow points to be related via the boundary.  There can be no infinite chains of curves with finitely many $g^1_4$'s, because such a fiber would be countable, and an infinite fiber must be uncountable.  Thus, infinite fibers will correspond to curves with infinitely many $g^1_4$'s and to singular curves.

\begin{theorem}[{\cite[Theorem IV.5.2]{MR770932},\cite[Theorem in Appendix]{MR0379510}}]
\label{thm:g14}
{\ \\}
Let $C$ be a smooth non-hyperelliptic curve of genus $g\geq 4$.  Let $d,r\in\ZZ$ with $0<2r\leq d$ and $2\leq d\leq g-2$.  Then, if there exists a $(d-2r-1)$-dimensional family of $g^r_d$'s on $C$, we must have that $C$ is trigonal, bielliptic or a plane quintic.
\end{theorem}

Fixing $g=6$, $d=4$ and $r=1$, the hypotheses on Theorem \ref{thm:g14} are satisfied, and so a curve of genus 6 that has a positive dimensional family of $g^1_4$'s must be hyperelliptic, trigonal, bielliptic, a plane quintic, or singular.

Proposition \ref{Hyperprym} says that Pryms of hyperelliptic curves are products.  Theorem \ref{Trigconst} tells us that trigonal Pryms are Jacobians.  Propositions \ref{Cubicfiber} and \ref{Jacfiber} tell us that Pryms of quintics are Jacobians or cubic threefolds and Proposition \ref{BoundaryPrym} says that the Prym varieties of singular curves are contained in $\J_5$, $\partial\At_5$, the Beauville Pryms or the product loci.

But we actually only care about contracted loci over $\alpha_4(\A_4)$, which may be irreducible components of $\RS_5$.  As $\dim(\alpha_4(\A_4))=10$, any irreducible component of $\RS_5$ must have dimension at least 10.  Also, as we will see later, $\overline{\RJ}_5$, $\overline{\RC}^0$, and $\partial^I\overline{\RA}_5$ are the only components of $\RS_5$ with positive local degree, thus, any other component is actually a contracted locus.

\begin{proposition}
\label{prop:contract}
{\ \\}
The only component of $\beta^{-1}_5(\alpha_4(\A_4))$ that is blown down by $\beta$ is $\RA_1\times \A_4$.
\end{proposition}

\begin{proof}
By the above discussion, the only loci with positive dimensional fibers are $\RJ_5$, the bielliptic loci, $\RC^0$, $\RC^1$, $\partial^I\overline{\RA}_5^t$, $\partial^{II}\overline{\RA}_5^t$, $\partial^{III}\overline{\RA}_5^t$, and $\RP_6(\partial^{III}\RRMbar_6)$.

The loci $\RJ_5$, $\RC^0$ and $\partial^I\overline{\RA}_5^t$ actually have positive local degree (see next chapter) and so are not blown down.

The bielliptic loci are ruled out by \ref{notbielliptic}.

The locus $\RC^1$ is not in $\RS_5$, as any point of $\RC^1$ will only be theta related to other points of $\RC^1$.  Specifically, a curve over $(X,\mu)\in\RC^1$ must be a plane quintic $Q$ and an odd point of order two $\nu$ with $P(Q,\nu)=X$, but also the lifts of $\mu$ must be odd, so all three points in any instance of theta symmetry give points of $\RC^1$.

From \ref{Schottkyboundary}, it is shown that $\partial^{III}\overline{\RA}^t_5$, the Beauville Pryms, $\RA_4\times \A_1$ and $\RA_4\times\RA_1$ are not in the Schottky locus, and all other product loci other than $\RA_1\times\A_4$ are of too small dimension.  However, $\RA_1\times\A_4$ is contracted by $\beta_5$.

The only remaining component is $\partial^{II}\overline{\RA}_5^t$, but we know this locus is mapped to $\PP^7/G_3$, and so any points of it in $\RS_5$ must be the limits of points of other components.
\end{proof}

\section{Schottky-Jung locus}

Finally, we prove our main theorem:

\begin{theorem}
\label{thm:RS5}
{\ \\}
In $\overline{\RA}^t_5$, we have $\overline{\RS}_5=\overline{\RJ}_5\cup \overline{\RC^0}\cup \partial^I\overline{\RA}^t_5\cup\A_4\times\RA_1$.
\end{theorem}

\begin{proof}
As the general fiber of $\beta$ is a connected $Q_6^-$-configuration, we have that $\deg\beta=119$.  Then Proposition \ref{prop:contract} tells us that the only blown down component of $\RS_5$ is $\A_4\times \RA_1$, so it remains to identify the components with positive local degree.

It is known Proposition \ref{genus5maps} that $\beta_5$ restricts to $\chi^{-1}:\RC^0\dashrightarrow \A_4$, and thus has local degree one.

The Schottky-Jung relations imply that the restriction to $\overline{\RJ}_5$ is birational to the Prym map $\RM_5\to\A_4$.  Thus, over $A\in\A_4$, we get a double cover of the Fano surface of a cubic threefold.  However, when we blow up to obtain a finite map, we get a double cover of the 27 lines on a cubic surface, and so we have local degree 54.

In \cite{MR844633}, van Geemen and van der Geer compute the local degree on $\partial^I\overline{\RA}_5^t$ as follows: fix $X\in\A_4$.  Then the part of the fiber of $\beta_5$ is $K(X)$ and by blowing up, they showed that the degree of the map is the same as that of the map $K(X)\to \PP^4$ given by the linear system $\Gamma_{00}=\{s\in \Gamma(X,2\Theta)|\mult_0s\geq 4\}$.  We blow $X$ up at $0$, call the exceptional divisor $E$, and then this linear system is precisely $2\Theta-4E$ on the abelian variety.  Thus, the local degree is $\frac{1}{2}(2\Theta-4E)^4=64$.

This gives us degree $1+54+64=119$, and so there are no other components with positive local degree.
\end{proof}

And so, as noted in \cite{MR898055}, this implies

\begin{corollary}
{\ \\}
$\S_5^{\sm}=\J_5$
\end{corollary}

\begin{proof}
The points of $\S_5^{\sm}$ are just the abelian varieties $A$ such that $(A,\mu)\in\RS_5$ for all nonzero points of order two $\mu\in A$.  Among the components above, Jacobians have all nonzero points of order two, but the intermediate Jacobians of cubic threefolds only have even points of order two, the boundary component has only the vanishing cycle, and the product locus only has pullbacks of the points of order two on the elliptic curve.
\end{proof}

\bibliographystyle{alpha}
\bibliography{Schottky}
\end{document}